\documentclass[11pt,leqno,a4paper,oneside]{amsart}

\usepackage{amssymb,amsthm,enumerate}
\usepackage{wasysym}
\usepackage{array}
\usepackage{hyperref}                         
\usepackage[all]{xy}

\usepackage{tikz}                                 
\usetikzlibrary{angles}
\usetikzlibrary{decorations.markings,calc}
\usetikzlibrary{cd}

\usepackage{longtable}
\usepackage{mathrsfs}
\usepackage{hhline}
\usepackage{multicol}
\usepackage{multirow}
\usepackage{blkarray} 
\usepackage{lscape}
\usepackage{rotating}

\input xy
\xyoption{all}

\entrymodifiers={!!<0pt,0.7ex>+}  

\hypersetup{pdftex,                           
bookmarks=true,
pdffitwindow=true,
colorlinks=true,
citecolor=black,
filecolor=black,
linkcolor=black,
urlcolor=black,
hypertexnames=true}

\newcommand{\IF}{{\mathbb{F}}}

\newcommand{\IZ}{{\mathbb{Z}}}

\newcommand{\fp}{{\mathfrak{p}}}     

\newcommand{\cO}{{\mathcal{O}}}

\newcommand{\bB}{{\mathbf{B}}}
\newcommand{\bb}{{\mathbf{b}}}

\DeclareMathOperator{\Syl}{Syl}                  
\DeclareMathOperator{\Ind}{Ind}                  
\DeclareMathOperator{\tr}{tr}			

\DeclareMathOperator{\Irr}{Irr}

\DeclareMathOperator{\TS}{TS}

\DeclareMathOperator{\St}{St}

\DeclareMathOperator{\Br}{Br}
\DeclareMathOperator{\Triv}{Triv}

\DeclareMathOperator{\GL}{\operatorname{GL}}
\DeclareMathOperator{\SL}{\operatorname{SL}}
\DeclareMathOperator{\PSL}{\operatorname{PSL}}

\let\lra=\longrightarrow

\newtheorem{thm}{Theorem}[section]
\newtheorem{lem}[thm]{Lemma}

\newtheorem{nota}[thm]{Notation}

\newtheoremstyle{defnew}{4ex}{}{}{}{\bf}{}{.5em}{}
\theoremstyle{defnew}

\newtheorem{conv}[thm]{Convention}

\theoremstyle{remark}

\theoremstyle{theorem}

\begin{document}


	\title[Trivial source character tables of $\SL_2(q)$, Part II]{Trivial source character tables of $\SL_2(q)$, Part II
	}
	\dedicatory{Dedicated to the to memory of Christine Bessenrodt}
	
	\date{\today}

	\author{{Niamh Farrell and Caroline Lassueur}}
	\address{{\sc Niamh Farrell}, Institut f\"ur Algebra, Zahlentheorie und Diskrete Mathematik, Leibniz Universit\"at Hannover, Welfengarten~1, 30176 Hannover, Germany.}
	\email{farrell@math.uni-hannover.de}
	\address{{\sc Caroline Lassueur},  FB Mathematik, TU Kaiserslautern, Postfach 3049, 67653 Kaiserslautern, Germany}
	\email{lassueur@mathematik.uni-kl.de}

	
	\makeatletter
	\@namedef{subjclassname@2020}{%
		\textup{2020} Mathematics Subject Classification}
	\makeatother
	

	\begin{abstract}
		We compute the trivial source character tables (also called species tables of the trivial source ring) of the infinite family of  finite groups $\SL_{2}(q)$ for $q$ even, over a large enough field  of odd characteristic. This article is a continuation of our article \emph{Trivial source character tables of $\SL_{2}(q)$} where we considered, in particular, the case in which~$q$~is odd in non-defining characteristic.
	\end{abstract}

	
	\maketitle

	
	\pagestyle{myheadings}
	\markboth{N. Farrell, C. Lassueur}{Trivial source character tables of $\SL_2(q)$, part II}

\section{Introduction}
Let $G$ be a finite group, let $\ell$ be a prime number dividing $|G|$ and let $k$ be an algebraically closed field of characteristic $\ell$. 
Permutation $kG$-modules and their direct summands, the trivial source modules, are omnipresent in the modular representation theory of finite groups. They are, for example,  elementary building blocks for the construction and for the understanding of different categorical equivalences between block algebras, such as source-algebra equivalences, Morita equivalences with endo-permutation source, splendid Rickard equivalences, or $\ell$-permutation equivalences.
A deep understanding of the structure of these modules is therefore essential.
\par
The \emph{trivial source character table of $G$ at the prime $\ell$}, denoted by  $\Triv_{\ell}(G)$, is by definition the species table of the trivial source ring of $kG$ 
in the sense of~\cite{BP}. The present article is a sequel to \cite{BFL22}, in which B\"ohmler and the authors calculate the trivial source character tables for the special linear group $\SL_2(q)$ over the finite field~$\IF_{q}$ when $q$ and $\ell$ are odd and $\ell \nmid q$, and when $q$ is odd, $\ell=2$ and $\SL_2(q)$ has quaternion Sylow $2$-subgroups. 
We refer the reader to the latter article for a complete introduction to trivial source character tables. 
We emphasise here that these table encapsulate, in a very compact way, a lot of information about the ordinary and Brauer characters of the trivial source $kG$-modules, as well as of those of their Brauer quotients. 
\par
In this article we calculate the trivial source character tables of $\SL_2(2^f) = \PSL_2(2^f)$  in non-defining characteristic for any integer $f \geq 2$. Note that if  $f = 1$ then $\SL_2(2) \cong S_3$ and the trivial source character tables are easily calculated using elementary arguments (see e.g. \cite{BensonBookOld}). Our main results, the trivial source character table $\Triv_{\ell}(\SL_2(2^f))$, appear in Tables \ref{tab:lmidq-1T_i1} and \ref{tab:lmidq-1T_ii}  for $\ell \mid 2 ^f-1$ and in Tables \ref{tab:lmidq+1T_i1}, \ref{tab:lmidq+1T_ii} and \ref{tab:lmidq+1T_n+1n+1} for $\ell \mid 2^f+1$. 
\par
The character table of $\SL_2(2^f)$ and the block distributions are given in~\cite{Burkhardt}. However, it is more convenient for our purposes to interpret this data in terms of Harish-Chandra and Deligne-Lusztig induction, as~\cite{BonBook} does for $\SL_{2}(q)$ with $q$ odd. This done, one of the main issues we solve in this article is the explicit calculation of the Brauer correspondence in the normaliser $N$ of a Sylow $\ell$-subgroup of $\SL_2(2^f)$, and the explicit calculation of the Green  correspondents in $N$  of the trivial source $k\SL_2(2^f)$-modules. 
\par
The paper is organised as follows. In Section~\ref{sec:notanddefs} we recall the notation and definitions for trivial source character tables and for blocks with cyclic defect groups which were established in \cite{BFL22}. Section~\ref{sec:SL2(q)} contains notation and preliminary results on the structure of $\SL_2(2^f)$. The trivial source character tables are calculated in Section~\ref{sec:tschartablelmidq-1} for $\ell \mid 2^f-1$,  respectively Section~\ref{sec:tschartablelmidq+1}  for $\ell \mid 2^f+1$\,.

\section{Notation and Definitions}
\label{sec:notanddefs}

\subsection{General notation}
Throughout, unless otherwise stated, we adopt the notation and conventions given below.
We let~$\ell$ denote a prime number and  $G$ denote a finite group of order divisible by~$\ell$. We let $(K,\cO,k)$ be an $\ell$-modular system, where  $\cO$ denotes a complete discrete valuation ring of characteristic zero with unique maximal ideal $\frak{p}:=J(\cO)$, algebraically closed residue field~$k=\cO/\fp$ of characteristic~$\ell$, and field of fractions $K=\text{Frac}(\cO)$, which we assume to be large enough for $G$ and its subgroups.
\par
Given a positive integer~$n$, we denote by $C_{n}$ the cyclic group of order~$n$. By an $\ell$-block of $G$, we mean a block algebra of $kG$. 
We denote by $\Irr(G)$ (resp. $\Irr(\bB)$)  the set of irreducible $K$-characters of~$G$ (resp. of the block of $\cO G$ corresponding to the $\ell$-block $\bB$), and if $H$ is abelian then we write $H^\wedge:=\Irr(H)$.  We write $\bB_0(G)$ for the principal $\ell$-block of $G$. 
\par

For $R\in\{\cO,k\}$, $RG$-modules are assumed to be finitely generated left $RG$-lattices, that is,  free as $R$-modules, and  we let  $R$  denote the trivial $RG$-lattice. 
If $M$ is a $kG$-module and $Q\leq G$, then the Brauer quotient (or Brauer construction) of $M$ at $Q$ is the $k$-vector space $M[Q]:=M^{Q}\big/ \sum_{R<Q}\tr_{R}^{Q}(M^{R})$, 
where $M^{Q}$ denotes the fixed points of $M$ under $Q$ and $\tr_{R}^{Q}$  for $R<Q$ denotes  the relative trace map. 
This vector space has a natural structure of a $kN_{G}(Q)$-module, but also of a $kN_{G}(Q)/Q$-module, and is equal to zero if $Q$ is not an $\ell$-group.
We refer the reader to \cite{LinckBook, TheBook} for further standard notation and background results in the modular representation theory of finite groups.

\subsection{Trivial source character tables}

Given $R\in\{\cO,k\}$, an $RG$-lattice $M$ is called a \emph{trivial source} $RG$-lattice if it is isomorphic to an indecomposable  direct summand of an induced lattice $\Ind_{Q}^{G}{R}$, and if $Q$ is of minimal order subject to this property, then $Q$ is a vertex of~$M$. 
Any trivial source $kG$-module $M$ lifts in a unique way to a trivial source $\cO G$-lattice $\widehat{M}$ (see e.g. \cite[Corollary 3.11.4]{BensonBookI}) and we denote by $\chi^{}_{\widehat{M}}$ the $K$-character afforded by $\widehat{M}$.
Up to isomorphism, there are only finitely many trivial source $kG$-modules (see e.g. \cite[Proposition 2.2 (d)]{BFL22}) and we will study them vertex by vertex. We denote by $\TS(G;Q)$ the set of isomorphism classes of indecomposable trivial source $kG$-modules with vertex $Q$. 
We let $a(kG,\Triv)$ be the \emph{trivial source ring} of $kG$, which is defined to be the subring of the Grothendieck ring of $kG$ generated by the set of all isomorphism classes of indecomposable trivial source $kG$-modules.

By definition the \emph{trivial source character table of the group $G$ at  the prime $\ell$}, denoted $\text{Triv}_{\ell}(G)$,  is the species table (or representation table) of the trivial source ring of $kG$ in the sense of Benson and Parker; see~\cite{BP}. 
However, as in~\cite{BFL22}, we follow~\cite[Section 4.10]{LP} and consider $\text{Triv}_{\ell}(G)$ as the block square matrix defined by the following notational convention.

\begin{conv}\label{conv:tsctbl}%
First, fix  a set of representatives $Q_1,\ldots, Q_r$ ($r\in\mathbb{N}$) for the conjugacy classes of $\ell$-subgroups of $G$ where $Q_{1}:=\{1\}$ and $Q_{r}\in\Syl_{\ell}(G)$. For each $1\leq v\leq r$ set $N_{v}:=N_{G}(Q_{v})$, $\overline{N}_{v}:=N_{G}(Q_{v})/Q_{v}$.  
Then, for each pair $(Q_{v},s)$ with $1\leq v\leq r$ and $s\in [\overline{N}_{v}]_{\ell'}$ there is a ring homomorphism  
		\begin{center}
			\begin{tabular}{cccl}
				$\tau_{Q_{v},s}^{G}$\,:            &   $a(kG,\mbox{Triv})$      & $\lra$ &    $K$     \\
				&   $[M]$      & $\mapsto$     &   $\varphi^{}_{M[Q_{v}]}(s)$   
			\end{tabular}
		\end{center}
mapping the class of a trivial source $kG$-module $M$ to the value at $s$ of the Brauer character $\varphi^{}_{M[Q_{v}]}$ of the Brauer quotient $M[Q_{v}]$. 
For each $1\leq i,v\leq r$ define a matrix
	$$T_{i,v}:=\big[ \tau_{Q_{v},s}^{G}([M])\big]_{M\in \TS(G;Q_{i}), s\in [\overline{N}_{v}]_{\ell'} }\,.$$
	The trivial source character table of  $G$ at the prime $\ell$ is then the block matrix 
	$$\text{Triv}_{\ell}(G):=[T_{i,v}]_{1\leq i,v\leq r}\,.$$
	Moreover, the rows of $\text{Triv}_{\ell}(G)$ are labelled with the ordinary characters $\chi^{}_{\widehat{M}}$ instead of the isomorphism classes of  trivial source modules $M$ themselves.
\end{conv}

We note that the group $G$ acts by conjugation on the pairs $(Q_{v},s)$ and the values of $\tau_{Q_{v},s}^{G}$ do not depend on the choice of  $(Q_{v},s)$ in its $G$-orbit. We refer the reader to \S 2 of our previous paper~\cite{BFL22} for details and further properties of trivial source modules and trivial source character tables. 

\subsection{Blocks with cyclic defect groups}\label{ssec:cycbl}
In the cases considered in  this article, we need to describe trivial source modules lying in blocks with cyclic defect groups. Therefore, we recall the following essential notions about cyclic blocks. Further details can be found in the first part of our paper \cite{BFL22} and also in \cite{HL20}. 
\par
Given an $\ell$-block $\bB$ of $kG$ with a non-trivial cyclic defect group $D\cong C_{\ell^n}$ ($n\geq 1$), we let $D_{1}<D$ denote the subgroup of $D$ of order $\ell$, 
we let $e$ denote the inertial index of $\bB$. We write
$$\Irr(\bB)=\Irr'(\bB)\sqcup\{\chi_{\lambda} \mid \lambda\in\Lambda\}\,,$$
where $\Irr'(\bB)$ is the set of the non-exceptional $K$-characters (there are $e$ of them) of $\bB$ and $|\Lambda|= \frac{|D|-1}{e}$. 
If $|\Lambda|>1$, then  $\{\chi_{\lambda} \mid \lambda\in\Lambda\}$ is the set of exceptional $K$-characters of $\bB$, which  all restrict in the same way to the $\ell$-regular conjugacy classes of~$G$. Further, we set $\chi_{\Lambda}:=\sum_{\lambda\in\Lambda}\chi_{\lambda}$.
The Brauer tree of $\bB$ is  then the graph  $\sigma(\bB)$  with vertices labelled by $\Irr^{\circ}(\bB):=\Irr'(\bB)\sqcup\{\chi_{\Lambda}\}$ and edges  labelled by the simple $\bB$-modules. 
If $|\Lambda|>1$  the vertex corresponding to $\chi_{\Lambda}$ is called the \emph{exceptional vertex} and is indicated with a filled black circle in our drawings of $\sigma(\bB)$.  
\\

Vertices and sources of indecomposable modules are encoded in a source algebra of a block, and hence so are the trivial source $\bB$-modules. 
We recall that by the work of Linckelmann~\cite{Linck96}, a source algebra of $\bB$ is determined up to isomorphism of interior $D$-algebras by three parameters:
\begin{enumerate}[{\,\,\rm(1)}] \setlength{\itemsep}{2pt}
\item $\sigma(\bB)$, understood with its planar embedding;
\item a \emph{type function} associating a sign to each vertex in an alternating way as follows: if $x$ is a generator of $D_{1}$, then a vertex $\chi\in\Irr^{\circ}(\bB)$ of $\sigma(\bB)$ is said to be positive  if $\chi(x)>0$, whereas it is said to be negative if $\chi(x)<0$;
\item an indecomposable capped endo-permutation $kD$-module $W(\bB)$; more precisely letting $\bb$ be the Brauer correspondent of $\bB$ in $N_G(D_1)$, then  $W(\bB)$ is defined to be a source of the  simple $\bb$-modules.  
\end{enumerate}
It turns out that $\bB$ contains precisely $e$ trivial source $kG$-modules for each possible vertex $Q\leq D$. These trivial source modules are explicitly classified by~\cite[Theorem~5.3]{HL20} as a function of the three parameters above. We refer to \cite[Remark~2.2]{BFL22} for a summary of this classification, relevant to the cyclic blocks of $\SL_{2}(q)$.

\section{Structure and characters of $\SL_2(q)$ when $q$ is even}
\label{sec:SL2(q)}

From now on, and until the end of this article, we assume that $\ell\neq 2$. Moreover, we let  ${G := \SL_2(q) = \PSL_2(q)}$ be the special linear group of degree $2$ over the finite field $\IF_{q}$, where~${q=2^f}$ for some integer $f \geq 2$.  Given a positive integer $r$, let $\mu_{r}$ be the group of the $r$-th roots of unity in an algebraic closure $\IF$ of $\IF_{q}$. For a subset of $S\subseteq \IF^{\times}$, let $[S/\equiv]$ denote a set of representatives for the elements of $S$ up to inverse.\\

In this section, we collect all necessary information about $G$ and its subgroups needed in order to calculate the trivial source character tables of $G$ in cross-characteristic.
Our aim is to use notation analogous to that used in~\cite{BonBook}.  However, \cite{BonBook} cannot be cited directly as it is assumed throughout the book that $q$ is odd. 
We note that the notation and some arguments need small adjustments when $q$ is even. We also refer to the unpublished master thesis of Schulte~\cite{Schulte} where further details, but not all, can be found.

\subsection{Tori, centralisers and normalisers  of $\ell$-elements}
We let $T := \{\text{diag}(a,a^{-1})  \mid a \in \IF_q^\times\}$ be the maximally split torus of $G$ consisting of the diagonal matrices. Then, there exists an isomorphism 
\[ \mathbf{d}: \mu_{q-1}=\IF_{q}^{\times} \longrightarrow T, a\mapsto \text{diag}(a,a^{-1}).\]
Fixing an $\IF_{q}$-basis of $\IF_{q^{2}}$ induces a group isomorphism $\mathbf{d'}:\GL_{\IF_{q}}(\IF_{q^{2}})\lra\GL_{2}(\IF_{q})$. 
The image $T':= \mathbf{d'}(\mu_{q+1})$ of $\mu_{q+1}$ under this isomorphism is a non-split torus of $G$. 
We will identify $T$ with $\mu_{q-1}$ and~$T'$ with $\mu_{q+1}$ via $\mathbf{d}$ and $\mathbf{d}'$ without further mention. 
As $q$ is even, both $T$ and $T'$ are cyclic groups of odd order. We let $S_\ell$ and $S_{\ell}'$ denote Sylow $\ell$-subgroups of $T$ and $T'$ respectively, and thus  we have a decomposition into direct products
$ T = S_{\ell} \times T_{\ell'}$, and $ T' = S_{\ell}' \times T_{\ell'}'$.
\par 
Finally, we consider the Frobenius automorphism  $F: \IF_{q^2}  \rightarrow \IF_{q^2}, x \mapsto x^q$, which we see as an element of $\GL_{\IF_q}(\IF_{q^2})$. We fix 
\[
\sigma := \left( \begin{smallmatrix} 0 & 1 \\ 1 & 0 \end{smallmatrix} \right)\qquad\text{and}\qquad \sigma' := \mathbf{d'}(F)\,;
\]
which are clearly both of order~$2$.

The following two lemmas are well known and can be proved using elementary arguments similar to those used in \cite[Section 1.3 and 1.4]{BonBook}.

\begin{lem}{\ }
\label{lem:1}
	\begin{enumerate}[{\,\,\rm(a)}]
		\item If $g = \textbf{d}(a)$ with $a \in \mu_{q-1} \setminus \{1\}$ then $C_G(g) = T$. In particular, $C_{G}(T)=T$.
		\item If $g = \mathbf{d'}(\xi)$ with $\xi \in \mu_{q+1} \setminus \{1\}$ then $C_G(g) = T'$. In particular, $C_{G}(T')=T'$.
	\end{enumerate}
\end{lem}

\begin{lem}{\ }
	\label{lem:omnibus}
		\begin{enumerate}[{\,\,{\rm(a)}}]
			\item  If $\ell \mid q-1$, then $S_\ell\in\Syl_{\ell}(G)$ and  $N_{G}(Q) = N_G(T) =  \langle T, \sigma \rangle =: N$ for any $1\lneq Q\leq S_{\ell}$.
			\item  If $\ell \mid q+1$, then $S_\ell'\in\Syl_{\ell}(G)$ and $N_{G}(Q)=N_{G}(T')= \langle T', \sigma' \rangle =: N'$ for any  $1\lneq Q\leq S_\ell'$.
		\end{enumerate}
\end{lem}

\vspace{2mm}
\subsection{Characters and conjugacy classes of $G$}
The conjugacy classes, the ordinary characters, and the character table of $G$ were known to Schur, and are given in~\cite[Sections I and II]{Burkhardt}. We use notation analogous to that used in~\cite{BonBook} for the case where $q$ is odd, however, as this is more convenient for our purposes.
In order to do this, we fix the following.

\begin{nota}{\ }
\begin{itemize}
\item[$\bullet$] Let  $1_G$ and $\St$ denote the trivial character and the  Steinberg character of $G$, respectively. 
\item[$\bullet$] Let $R : \IZ \Irr(T) \longrightarrow \IZ \Irr(G)$ and $R' : \IZ \Irr(T') \longrightarrow \IZ \Irr(G)$ 
denote Harish-Chandra induction and Deligne-Lusztig induction, respectively.
\item[$\bullet$] Set  $\Gamma :=  [(\mu_{q-1} \setminus \{1\})/\equiv]$, $\Gamma' :=  [(\mu_{q+1} \setminus \{1\})/\equiv]$. 
\item[$\bullet$] Fix the following set of representatives for the conjugacy classes of $G$
\[ \{I_2\} ~ \dot \cup ~\{u\} ~ \dot \cup ~ \{ \mathbf{d}(a) \mid a \in \Gamma \} ~ \dot \cup ~ \{\mathbf{d}'(\xi) \mid \xi \in \Gamma' \}\,, \]
where $ u := \left( \begin{smallmatrix} 1	& 1 \\ 0 & 1 \end{smallmatrix} \right)$ is an element of order~$2$.
\end{itemize}
\end{nota}

There are $q-2$ non-trivial characters $\alpha \in \Irr(T)$, all satisfying $R(\alpha) = R(\alpha^{-1}) \in \Irr(G)$, giving us $\frac{q-2}{2}$ irreducible characters in $\Irr(G)$. Similarly, the $q$ non-trivial characters $\theta \in \Irr(T')$ satisfy $R'(\theta) = R'(\theta^{-1}) \in \Irr(G)$, giving us $\frac{q}{2}$ irreducible characters of $G$. Hence,
\[	\Irr(G) =    \{ 1_G, \St \} \cup   \{ R(\alpha) \mid \alpha \in [T^\wedge/\equiv], \alpha \neq 1\}  \cup   \{ R'(\theta) \mid \theta \in [T'^\wedge / \equiv] , \theta \neq 1\},
\]
and the character table of $G$ is as given in Table~\ref{tab:CTBLsl2}. 	

\renewcommand*{\arraystretch}{1.4}
\begin{longtable}{|c||c|c|c|c|}
	\caption{Character table of $\SL_2(q)$ when $q$ is even.}\label{tab:CTBLsl2}	
	\\  \hline 
	&
	\begin{tabular}{c}
		\textbf{$I_2$ } 
	\end{tabular}
	
	&
	\begin{tabular}{c}
		\textbf{$\textbf{d}(a)$} \\
		\textbf{	$ a \in \Gamma$}
	\end{tabular}
	&
	\begin{tabular}{c}
		\textbf{$\textbf{d}'(\xi)$} \\
		\textbf{	$\xi \in \Gamma'$}
	\end{tabular}	
& 
\begin{tabular}{c}
\textbf{$u$} 
\end{tabular}
	\\ \hline \hline 			
	
		No. of classes 
		& $1$
		& $\frac{q- 2}{2}$
		& $\frac{q}{2}$
		& $1$
	\\ \hline 

		Order of $g$
		& $1$
		& $o(a)$
		& $o(\xi)$
		& $2$
		\\ \hline 
		Class size
		& 1
		& $q(q+1)$
		& $q(q-1)$
		& $(q-1)(q+1)$
		\\ \hline \hline
%
%
	$1_G$
	& $1$
	& $1$
	& $1$
	& $1$
	\\ \hline

	$\St$
	& $q$
	& $1$
	& $-1$
	& $0$
	\\ \hline
	
	$R(\alpha) ~ (\alpha \in [T^\wedge / \equiv], \alpha \neq 1)$ 
	& $q+1$
	& $\alpha(a) + \alpha(a^{-1})$
	& $0$
	& $1$ 
	\\ \hline
	
	$R'(\theta) ~ (\theta \in [T'^\wedge / \equiv], \theta \neq 1)$ 
	& $q-1 $
	& $0$
	& $-\theta(\xi) - \theta(\xi^{-1})$
	& $-1$
	\\ \hline
\end{longtable}

\enlargethispage{15mm}

\vspace{2mm}
\subsection{Characters and conjugacy classes of $N$ and $N'$}
\label{subsec:NandN'}
 
We adopt here notation for the character theory of $N$ and $N'$ analogous to the notation used in \cite[Sections~6.2.1 and~6.2.2]{BonBook} for the case in which $q$ is odd. First, we fix the following sets of representatives for the conjugacy classes of $N$ and $N'$, respectively. 
\[ \{I_2\} \cup \{\sigma\}\cup \{ \mathbf{d}(a) \mid a \in \Gamma \} \qquad \qquad \qquad \{I_2\} \cup \{\sigma'\}\cup \{ \mathbf{d'}(\xi) \mid \xi \in \Gamma' \}\]
The difference with the odd case is that when $q$ is even,  $T$ has no non-trivial $N$-invariant characters and $T'$ has no non-trivial $N'$-invariant characters. For $\alpha\in\Irr(T)\setminus\{1\}$ (respectively,  $\theta\in \Irr(T')\setminus\{1\}$), we let $\chi_{\alpha}$ be the unique element of  $\Irr(N)$ such that $\chi_\alpha = \Ind_T^N (\alpha) = \Ind_{T}^N (\alpha^{-1})$ (respectively we let $\chi'_{\theta}$  be the unique element of $\Irr(N')$ such that $\chi'_{\theta} = \Ind_{T'}^{N'}(\theta) = \Ind_{T'}^{N'}(\theta^{-1})$). We let $\varepsilon$ (respectively $\varepsilon'$) denote the  linear character of $N$ (respectively of $N'$) of order $2$. With this notation, the character tables of $N$ and $N'$ are as follows.

\renewcommand*{\arraystretch}{1.4}
\begin{longtable}{|c||c|c|c|c|}
	\caption{Character table of $N$.}\label{tab:CTBLN}		
	\\  \hline 
	&
	\begin{tabular}{c}
		\textbf{$I_2$ } 
	\end{tabular}
	&
	\begin{tabular}{c}
		\textbf{$\textbf{d}(a)$} \\
		\textbf{	$ a \in \Gamma$}
	\end{tabular}
	&
	\begin{tabular}{c}
		\textbf{$\sigma$} 
	\end{tabular}	
	
	\\ \hline \hline 		
	
	No. of classes 
	& $1$
	& $\frac{q-2}{2}$
	& $1$
	\\ \hline 
	
	Order of $g$
	& $1$
	& $o(a)$
	& $2$
	\\ \hline 
	
	Class size
	& $1$
	& $2$
	& $q-1$
	\\ \hline 
	
	$C_N(g)$
	& $N$
	& $T$
	& $\langle \sigma \rangle$
	\\ \hline \hline 
	
	$1_N$
	& $1$
	& $1$
	& $1$
	\\ \hline

	$\varepsilon$
	& $1$
	& $1$
	& $-1$
	\\ \hline
	
	$\chi_\alpha$
	& $2$
	& $\alpha(a) + \alpha(a^{-1})$
	& $0$
	\\ \hline
\end{longtable}

\renewcommand*{\arraystretch}{1.4}
\begin{longtable}{|c||c|c|c|c|}
	\caption{Character table of $N'$.}\label{tab:CTBLN'}	
	\\  \hline 
	&
	\begin{tabular}{c}
		\textbf{$I_2$ } 
	\end{tabular}
	&
	\begin{tabular}{c}
		$\mathbf{d'}(\xi)$ \\
		$ \xi \in \Gamma'$
	\end{tabular}
	&
	\begin{tabular}{c}
		\textbf{$\sigma'$} 
	\end{tabular}

	\\ \hline \hline 			
	
	No. of classes 
	& $1$
	& $\frac{q}{2}$
	& $1$
	\\ \hline 
	
	Order of $g$
	& $1$
	& $o(\xi)$
	& $2$
	
	\\ \hline 
	
	Class size
	& $1$
	& $2$
	& $q+1$
	
	\\ \hline 
	
	$C_{N'}(g)$
	& $N'$
	& $T'$
	& $\langle \sigma' \rangle$
	
	\\ \hline \hline

	$1_{N'}$
	& $1$
	& $1$
	& $1$
	\\ \hline

	$\varepsilon'$
	& $1$
	& $1$
	& $-1$
	
	\\ \hline
	
	$\chi'_\theta$
	& $2$
	& $\theta(\xi) + \theta(\xi^{-1})$
	& $0$
	
	\\ \hline
\end{longtable}

\vspace{3mm}
\section{Trivial source character table of $G$ when $\ell \mid q-1$}
\label{sec:tschartablelmidq-1}

\begin{nota}\label{rem:lmidq-1}
	
	In order to describe $\Triv_{\ell}(G)$ according to Convention~\ref{conv:tsctbl} we adopt the following notation. We fix $Q_{n+1}:=S_{\ell}\cong C_{\ell^{n}}$ and for each $1\leq i\leq n$ we let $Q_{i}$ denote the unique cyclic subgroup of $Q_{n+1}$ of order $\ell^{i-1}$. The chain of subgroups
	\[\{1\} = Q_1 \leq \dots  \leq Q_{n+1} \in \Syl_{\ell}(G)\]
	is then our fixed set of representatives for the conjugacy classes of $\ell$-subgroups of $G$. 
	
	We fix $\Gamma_{\ell'} = [ ((\mu_{q-1})_{\ell'} \setminus \{1\})/ \equiv]$ and $\Gamma'_{\ell'} = [ ((\mu_{q+1})_{\ell'} \setminus \{1\})/ \equiv]$. Note that here $\Gamma'_{\ell'} = \Gamma'$ as $\ell \nmid q+1$.  We fix the following set of representatives for  the $\ell'$-conjugacy classes of $G$:
	$$[G]_{\ell'}:=\{I_{2}\}\cup\{ u \} \cup  \{\mathbf{d}(a)\mid a\in\Gamma_{\ell'}\}\cup  \{\mathbf{d}'(\xi)\mid \xi\in\Gamma'_{\ell'}\}\,.$$
	For any $2\leq v\leq n+1$, $1\leq i\leq n+1$, the columns of $T_{i,v}$ are labelled by a set of representatives for the $\ell'$-conjugacy classes of $\overline N_{v} = N_G(Q_{v})/Q_{v}=N/Q_{v}$ as  $N_G(Q_v) = N_G(T) = N$ for each $2\leq v \leq n+1$ by Lemma~\ref{lem:omnibus}(a).  However, since $Q_{v}$ is an $\ell$-group we will simply label the columns of $T_{i,v}$ by the following fixed set of representatives for  the $\ell'$-conjugacy classes of $N$
	$$[N]_{\ell'}:=\{I_{2}\}\cup\{ \sigma \} \cup  \{\mathbf{d}(a)\mid a\in\Gamma_{\ell'}\}\,.$$	
	Moreover, in order to describe the exceptional characters occurring  as constituents of the trivial source characters, for each $0\leq  i \leq n$ we fix
	\[
	\pi_{q,i} :=\frac{(q-1)_\ell\cdot \ell^{-i} - 1}{2}\,,
	\]
	we let $\pi_q := \pi_{q,0}$, and note that 
	$\pi_{q,n}=0$. These numbers arise naturally from the classification of the trivial source modules in cyclic blocks in \cite{HL20}. 
\end{nota}

\vspace{2mm}
\subsection{The $\ell$-blocks and trivial source characters of $G$}
\label{subsec:Glmidq-1}

\begin{lem} {\label{lem:blockslmidq-1}}
	When $\ell \mid q-1$ the $\ell$-blocks of $G$, their defect groups and their Brauer trees with type function are as given in Table~\ref{tab:blockslmidq-1}.
\end{lem}

\begin{longtable}{c||c|c|c}
	\caption{The $\ell$-blocks of $\SL_2(q)$ when $\ell \mid q-1$ and $q$ is even.}	
	\label{tab:blockslmidq-1}
	\\  
	Block $\bB$
	&	\begin{tabular}{c}
		Number of Blocks \\
		(Type)
	\end{tabular}
	& 	\begin{tabular}{c}
		Defect \\
		Groups
	\end{tabular}
	& 
	\begin{tabular}{c}
		Brauer Tree with Type  \\
		Function or $\Irr(\bB)$
	\end{tabular}
	\\ \hline \hline 			
	
	$\mathbf{B}_0(G)$ 
	& 	\begin{tabular}{c}
		1 \\
		\small{(Principal)} 
	\end{tabular} 
	& 	$C_{\ell^n}$ 
	&	\begin{tabular}{c}
		$$ \xymatrix@R=0.0000pt@C=30pt{
			{_+} & {_-}& {_+}\\
			{\Circle} \ar@{-}[r] 
			& {\CIRCLE} \ar@{-}[r] 
			&{\Circle} \\
			{^{1_G}}&{^{\Xi}}&{^{\St}}
		}$$\\
		{\footnotesize $\Xi:=\sum\limits_{\eta \in [S_{\ell}^\wedge / \equiv], \eta \neq 1}R(\eta)$}\\
	\end{tabular}
	\\ \hline 
	
	\begin{tabular}{c}
		$A_\alpha$ \\
		\footnotesize{($\alpha \in [T_{\ell'}^\wedge / \equiv], \alpha \neq 1$)}
	\end{tabular}
	& 	\begin{tabular}{c}
		$\frac{(q-1)_{\ell'} - 1}{2}$\\
		\small{(Nilpotent)} 
	\end{tabular}	
	& 	$C_{\ell^n}$ 		
	& 	\begin{tabular}{c}
		$  \xymatrix@R=0.0000pt@C=30pt{	
			{_+} & {_-}\\
			{\Circle} \ar@{-}[r] 
			& {\CIRCLE}    \\
			{^{R(\alpha)}}&{^{\Xi_\alpha}}
		}$ \\
		{\footnotesize $\Xi_{\alpha}  := \sum\limits_{\eta \in S_{\ell}^\wedge \setminus \{1\}} R(\alpha\eta)$}\\
	\end{tabular}
	\\ \hline 
	
	\begin{tabular}{c}
		$A'_\theta$  \\
		\footnotesize{($\theta \in [T_{\ell'}'^\wedge / \equiv], \theta \neq 1)$}
	\end{tabular}
	& 	\begin{tabular}{c}
		$\frac{(q+1)_{\ell'} - 1}{2} = \frac{q}{2}$ \\
		\small{(Defect zero)}
	\end{tabular} 
	& 	$\{1\}$ 
	& 	$\Irr(A'_\theta) = \{R'(\theta)\}$
	\\ \hline 
\end{longtable}

\begin{proof}
	All of the information in the table comes directly from \cite[Section I]{Burkhardt} and the character table of $G$ (Table~\ref{tab:CTBLsl2}), except for the type functions on the Brauer trees, which we compute according to (2) in \S\ref{ssec:cycbl}.  
	The trivial character is clearly positive so the type function for the principal block is immediate. For each block $A_{\alpha}$, the $\ell'$-character $\alpha$ takes the value $1$ on $\ell$-elements and therefore $R(\alpha)$ is positive.
\end{proof}

\enlargethispage{20mm}

\begin{lem}
	\label{lem:tsmodslmidq-1}
	When $\ell \mid q-1$ the ordinary characters $\chi_{\widehat{M}}$ of the trivial source $kG$-modules~$M$  are as given in Table~\ref{tab:tsmodslmidq-1}, where for each $1\leq i\leq n$, 
	\[\Xi_i = \sum\limits_{j=1}^{\pi_{q,i}} R(\eta_j)\]
	is a sum of  $\pi_{q,i}$ pairwise distinct exceptional characters  in $\bB_{0}(G)$, and  for any non-trivial character {$\alpha \in \Irr(T_{\ell'})$}, 
	\[\Xi_{\alpha,i} = \sum\limits_{j=1}^{2\pi_{q,i}} R(\alpha\eta_j)\]
	is a sum of $2\pi_{q,i}$ pairwise distinct exceptional characters in $A_{\alpha}$. 
\end{lem}

\begin{longtable}{|c|c|c|}
	\caption{Trivial source characters of $\SL_2(q)$ when $\ell \mid q-1$.}	
	\label{tab:tsmodslmidq-1}
	\\  \hline 
	\begin{tabular}{c}
		Vertices of $M$
	\end{tabular}
	& \begin{tabular}{c}
		Character $\chi_{\widehat{M}}$
	\end{tabular}
	& \begin{tabular}{c}
		Block containing $M$
	\end{tabular}
	\\ \hline \hline 
	
	\multirow{3}{*}{$\{1\}$}
	&$1_G + \Xi$,  
	$\St + ~ \Xi$
	&
	$\mathbf{B}_0(G)$
	\\

	&
	$R(\alpha) + \Xi_{\alpha}$
	&
	$A_{\alpha}$ ($\alpha \in [T_{\ell'}^\wedge / \equiv], \alpha \neq 1$)
	\\
	
	&
	$R'(\theta)$
	&
	$A'_{\theta}$ ($\theta \in [T_{\ell'}'^\wedge / \equiv], \theta \neq 1$)
	\\
	\hline \hline

	\multirow{2}{*}{
		\begin{tabular}{c}
			$C_{\ell^i}$ \\
			($1 \leq i < n$)
	\end{tabular}}
	&
	$1_G + \Xi_i$, $\St + ~ \Xi_i$
	&
	$\mathbf{B}_0(G)$
	\\

	&
	$R(\alpha) + \Xi_{\alpha, i}$
	&
	$A_{\alpha}$ ($\alpha \in [T_{\ell'}^\wedge / \equiv], \alpha \neq 1$)
	\\
	\hline \hline
	
	\multirow{2}{*}{
		$C_{\ell^n}$}
	&
	$1_G$, $\St$
	&
	$\mathbf{B}_0(G)$
	\\

	&
	$R(\alpha)$
	&
	$A_{\alpha}$ ($\alpha \in [T_{\ell'}^\wedge / \equiv], \alpha \neq 1$)
	\\
	\hline 	
\end{longtable}

\begin{proof}
First, the ordinary characters of the PIMs lying in blocks of defect zero are immediate from Table~\ref{tab:blockslmidq-1}, and the characters of the PIMs lying in blocks with a non-trivial cyclic defect group can also be read off from Table~\ref{tab:blockslmidq-1}, e.g. using~\cite[Remark~2.6(a)]{BFL22}.
\par
The trivial source $kG$-modules with non-trivial vertices $C_{\ell^{i}}$ ($1\leq i\leq n$) all belong to $\ell$-blocks $\bB$ with a non-trivial cyclic defect group.
By~\S\ref{ssec:cycbl}, each such block contains precisely $e$ trivial source $kG$-modules with vertex $C_{\ell^{i}}$, where $e$ is the inertial index of the block. Moreover, in order to determine these modules up to isomorphism, we need parameters (1), (2) and (3) of~\S\ref{ssec:cycbl}, namely the Brauer trees with their type function, which are given in~Table~\ref{tab:blockslmidq-1}, and the module $W(\bB)$, which is always trivial in our case by~\cite[Proposition~6.5(a)]{HL22}.
Thus, the characters $\chi^{}_{\widehat{M}}$ listed in Table~\ref{tab:tsmodslmidq-1} are obtained by applying the classification of the trivial source modules given in  \cite[Theorem~5.3(b)(2) and Theorem~A.1(d)]{HL20}, exactly as in~\cite[Lemma~3.3]{BFL22}.
\end{proof}

\vspace{4mm}
\subsection{The $\ell$-blocks and trivial source characters of $N$}
\label{subsec:N}

\begin{lem}
	\label{lem:blocksN}
	When $\ell \mid q-1$ the $\ell$-blocks of $N$, their defect groups and their Brauer trees with type function are as given in Table \ref{tab:blocksN}.
\end{lem}

\begin{longtable}{c||c|c|c}
	\caption{The $\ell$-blocks of $N$ when $\ell \mid q-1$ and $q$ is even.}	
	\label{tab:blocksN}
	\\  
	\begin{tabular}{c}
		Block \\
		$\bb$
	\end{tabular}
	& 	
	\begin{tabular}{c}
		Number of Blocks \\
		(Type)
	\end{tabular}
	&
	\begin{tabular}{c}
		Defect \\
		Groups
	\end{tabular}
	& 
	\begin{tabular}{c}
		Brauer Tree $\sigma(\bb)$ \\
		with Type Function
	\end{tabular}
	\\ \hline \hline 			
	
	$\bB_0(N)$  
	& 	\begin{tabular}{c}
		$1$ \\
		(Principal)
	\end{tabular}
	& 	$C_{\ell^n}$ 
	&	\begin{tabular}{c}
		$$ \xymatrix@R=0.0000pt@C=30pt{
			{_+} & {_-}& {_+}\\
			{\Circle} \ar@{-}[r]
			& {\CIRCLE} \ar@{-}[r] 
			&{\Circle} \\
			{^{1_N}}&{^{\Xi^N}}&{^{\varepsilon}}
		}$$\\
		{\footnotesize $\Xi^{N} := \sum\limits_{\eta \in [S_{\ell}^\wedge / \equiv], \eta \neq 1} \chi_{\eta}$}\\
	\end{tabular}
	\\ \hline

	\begin{tabular}{c}
		$\bb_\alpha$ \\
		$(\alpha \in [T_{\ell'}^\wedge / \equiv], \alpha \neq 1)$
	\end{tabular}
	&
	\begin{tabular}{c}
		$\frac{(q-1)_{\ell'} - 1}{2}$ \\
		(Nilpotent)
	\end{tabular}
	& 	$C_{\ell^n}$ 		
	& 	\begin{tabular}{c}
		$  \xymatrix@R=0.0000pt@C=30pt{	
			{_+} & {_-}\\
			{\Circle} \ar@{-}[r] 
			& {\CIRCLE}    \\
			{^{\chi_\alpha}}&{^{\Xi^{N}_{\alpha}}}
		}$\\
		{\footnotesize	$\Xi^{N}_{\alpha}  := \sum\limits_{ \eta \in S_{\ell}^\wedge \setminus \{1\}} \chi_{\alpha\eta}$ }\\  
	\end{tabular}
	\\ \hline 
	
\end{longtable}
\vspace{10mm}

\begin{proof}
	We determine the partition of $\Irr(N)$ into $\ell$-blocks of $N$ by examining the central characters of $N$ modulo $\ell$ and we find that:  
	\begin{itemize}
		\item $\Irr(\bB_0(N))$ contains $1_N$, $\varepsilon$ and $\{{\chi_{\eta}\mid \eta \in \Irr(S_{\ell}) \setminus \{1\}}\}$; and
		\item for each non-trivial $\alpha \in \Irr(T_{\ell'})$ there exists a block $\bb_\alpha$ containing $\chi_{\alpha \eta}$ for all $\eta \in \Irr(S_{\ell})$.
	\end{itemize}
	Since all the blocks have maximal normal defect groups, their Brauer trees are star-shaped with a central exceptional vertex (see e.g. \cite[Proposition 6.5.4]{BensonBookI}). The Brauer trees are therefore fully determined because the $\ell$-rational characters $1_N$, $\varepsilon$ and $\chi_\alpha$ ($\alpha \in \Irr(T_{\ell'})$) must be non-exceptional. The type functions, as defined in (2) of \S\ref{ssec:cycbl}, are immediate in this case, as the cyclic subgroups of order $\ell$ of the defect groups are normal in $N$ by Lemma~\ref{lem:omnibus}(a).	
\end{proof}

\begin{lem}
	\label{lem:tsmodsN}
	When $\ell \mid q-1$ 
	the ordinary characters $\chi_{\widehat{f(M)}}$ of the $kN$-Green correspondents $f(M)$ of  the trivial source $kG$-modules $M$ with a non-trivial vertex  are as given  in Table~\ref{tab:tsmodsN}, 
	where, for each $1\leq i\leq n$, 
	\[\Xi^{N}_i := \sum\limits_{j=1}^{\pi_{q,i}} \chi_{\eta_j}\]
	is a sum of  $\pi_{q,i}$ pairwise distinct exceptional characters  in $\Irr(\bB_{0}(N))$, and, for any non-trivial $\alpha \in \Irr(T_{\ell'})$, 
	\[\Xi^{N}_{\alpha, i} := \sum\limits_{j=1}^{2\pi_{q,i}} \chi_{\alpha\eta_j}\]
	is a sum of $2\pi_{q,i}$ pairwise distinct exceptional characters  in 
	$\Irr(\mathbf{b}_\alpha)$. 
\end{lem}

\begin{longtable}{|c|l|l|}
	\caption{The trivial source characters of the $kN$-Green correspondents when $\ell \mid q-1$.}	
	\label{tab:tsmodsN}
	\\  \hline 
	\begin{tabular}{c}
		Vertices of $M$
	\end{tabular}
	& \begin{tabular}{c}
		Character $\chi_{\widehat{M}}$
	\end{tabular}
	& \begin{tabular}{c}
		Character $\chi_{\widehat{f(M)}}$ of the  \\Green Correspondent 
	\end{tabular}
	\\ \hline \hline

	
	\begin{tabular}{c}
		$C_{\ell^i}$ \\
		($1 \leq i < n$)
	\end{tabular} 
	& 	\begin{tabular}{l}
		$1_G + \Xi_i$\\ 
		$\St + \Xi_i$\\ 
		$R(\alpha) + \Xi_{\alpha, i}$ ($\alpha \in [T_{\ell'}^\wedge / \equiv], \alpha \neq 1$) 
	\end{tabular}
	& \begin{tabular}{l}
		$1_N + \Xi^N_i$\\ 
		$\varepsilon + \Xi^N_i$\\ 
		$\chi_\alpha + \Xi^{N}_{\alpha, i}$
	\end{tabular}
	
	\\ 	\hline 	 \hline 
	
	$C_{\ell^n}$
	& 
	\begin{tabular}{l}
		$1_G$\\
		$\St$\\
		$R(\alpha)$ ($\alpha \in [T_{\ell'}^\wedge / \equiv], \alpha \neq 1$)
	\end{tabular}
	&
	\begin{tabular}{l}
		$1_N$ \\
		$\varepsilon$ \\
		$\chi_{\alpha}$
	\end{tabular}
	\\  \hline 
	
\end{longtable}

\begin{proof}	
	We first determine the Brauer correspondents in $N$ of the blocks of $G$. 
	We claim that for a fixed $\alpha \in T_{\ell'}^\wedge$, $\alpha \neq 1$, the block $\bb_\alpha$ is the Brauer correspondent in $N$ of the $A_\alpha$. Let $i_\alpha$ be the central primitive idempotent of $\mathcal{O}G$ such that $\mathcal{O}Gi_\alpha$ is the block of $\mathcal{O}G$ corresponding to the $\ell$-block $A_\alpha$. Thus,
	\[ i_\alpha : = \sum_{\eta \in S_{\ell}^\wedge}  \sum_{g \in G} \frac{1}{|G|} R(\alpha \eta)  (1) R(\alpha \eta)(g) g^{-1}
	= \sum_{g \in G}  \frac{q+1}{|G|} \sum_{\eta \in S_{\ell}^\wedge} R(\alpha \eta)(g) g^{-1}
	\]
	When we apply the Brauer homomorphism to $\bar i_\alpha$, the image of $i_\alpha$ in $kG$, the only terms in the sum which survive are those for $g \in C_{G}(S_\ell) = T$. The coefficient in $i_\alpha$ of a non-trivial element $\mathbf{d}(a^{-1}) \in T$ for some $a \in \mu_{q-1} \setminus \{1\}$ is 
	\begin{align*}
		\frac{q+1}{|G|} \sum_{\eta \in S_{\ell}^\wedge} R(\alpha \eta) (\mathbf{d}(a)) 
		& = \frac{q+1}{|G|} \left(\alpha(a)\sum_{\eta \in S_{\ell}^\wedge} \eta(a) + \alpha(a^{-1})\sum_{\eta \in S_{\ell}^\wedge}\eta(a^{-1}) \right). 
	\end{align*}
	If $a$ has non-trivial $\ell$-part then the second orthogonality relations show that ${\sum_{\eta \in S_{\ell}^\wedge} \eta(a) = 0}$. Thus the only elements in $T$ with non-zero coefficients in $i_\alpha$ are the $\ell'$-elements and they have coefficients
	\[
	\left\{	
	\begin{array}{ll}
		\frac{(q+1)^2 |S_{\ell}|}{|G|} = \frac{q+1}{q |T_{\ell'}|}
		& \mbox{if } a = 1, \mbox{and} \\
		\frac{(q+1)|S_{\ell}|}{|G|} \left(\alpha(a)+ \alpha(a^{-1}) \right) = \frac{1}{q |T_{\ell'}|} \left(\alpha(a)+ \alpha(a^{-1}) \right)
		& \mbox{if } 1 \neq a \in (\mu_{q-1})_{\ell'}.
	\end{array} \right. 
	\]
	
	Now let $i^N_\alpha$ denote the central primitive idempotent of $\mathcal{O}N$ such that $\mathcal{O}Ni^N_\alpha$ is the block of $\mathcal{O}N$ corresponding to the $\ell$-block $\bb_\alpha$. Then 
	\[ i_\alpha^N = \sum_{\eta \in S_{\ell}^\wedge}  \sum_{n \in N} \frac{1}{|N|} \chi_{\alpha \eta}  (1) \chi_{\alpha \eta} (n) n^{-1}
	= \sum_{n \in N}  \frac{2}{|N|} \sum_{\eta \in S_{\ell}^\wedge}  \chi_{\alpha \eta} (n) n^{-1}
	.\]
	Since $\chi_{\alpha\eta}$ is linear and $\chi_{\alpha\eta} (\sigma) = 0$, in fact this sum only has non-zero terms for $n \in T$. By the same arguments as above, the only elements $\mathbf{d}(a^{-1}) \in T$ with non-zero coefficients in $i^N_\alpha$ are the $\ell'$-elements and they have coefficients
	\[
	\left\{	
	\begin{array}{ll}
		\frac{4|S_{\ell}|}{|N|} = \frac{2}{|T_{\ell'}|}
		& \mbox{if } a = 1, \mbox{and}\\
		\frac{2|S_{\ell}|}{|N|} \left( \alpha(a) + \alpha(a^{-1})\right) = \frac{1}{|T_{\ell'}|}\left( \alpha(a) + \alpha(a^{-1})\right)
		& \mbox{if } 1 \neq a \in (\mu_{q-1})_{\ell'}.
	\end{array} \right. 
	\]
	
	Let $\bar i^N_\alpha$ denote the image of $i_\alpha^N$ in $kN$. 
	Then since $\ell \mid q-1$, we have $q \equiv 1 \mod \mathfrak{p}$. Therefore $ \bar i^N_\alpha \equiv \Br_{S_{\ell}}(\bar i_\alpha) \mod \mathfrak{p}$, where $\Br_{S_{\ell}}$ denotes the Brauer homomorphism. In particular, $\bb_\alpha$ is the Brauer correspondent in $N$ of $A_\alpha$. 
	\par  
	Next, we recall that the Brauer correspondence and the Green correspondence commute, so if a trivial source $kG$-module lies in the block $\bB$, then its Green correspondent lies in the Brauer correspondent $\bb$ of $\bB$. Moreover, by definition, $W(\bb)=W(\bB)$, which as we already noticed is trivial in all cases. 
	Therefore the characters of the trivial source $\bb$-modules can be determined using \cite[Theorem~5.3]{HL20} in all cases, as we did for $G$. This yields the list of characters in the third column of Table~\ref{tab:tsmodsN}, up to reordering. Therefore, it only remains to check that the characters printed on the same lines of the second and third columns are the characters of Green correspondent modules. For $\bB_{0}(G)$ it is enough to notice that if a trivial source module of $kG$ has the trivial character as a constituent of its ordinary character, then so does its $kN$-Green correspondent. Since both $A_\alpha$ and $\bb_\alpha$ have to contain a unique trivial source module with a given vertex, there is only one possibility for the blocks of type $A_\alpha$, as required. 
\end{proof}

\color{black}

\vspace{4mm}

\subsection{The trivial source character table of $G$}

\begin{thm}\label{thm:l|q-1} 
Let $G= \SL_2(q)$ with $q = 2^f$ for an integer $f \geq 2$ and suppose that $\ell \mid q-1$.  
Then, with notation as in Notation~\ref{rem:lmidq-1}, 
	the trivial source character table $\Triv_{\ell}(G)= [T_{i,v}]_{1\leq i,v\leq n+1}$ is given as follows:
	\begin{enumerate}[{\,\,\rm(a)}] \setlength{\itemsep}{2pt}
		\item $T_{i,v} = \mathbf{0}$ if $v > i$;
		\item the matrices $T_{i,1}$ are as given in Table \ref{tab:lmidq-1T_i1} for each $1 \leq i \leq n+1$;
	        \item the matrices $T_{i,i}$ are as given in Table \ref{tab:lmidq-1T_ii} for each $2 \leq i \leq n+1$; and 
		\item $T_{i,v} = T_{i,i}$ for all $2 \leq v < i \leq n+1$.
	\end{enumerate} 
\end{thm}

\begin{landscape}
\thispagestyle{plain}
	
	\renewcommand*{\arraystretch}{1.5}
	\begin{longtable}{c|c||c|c|c|c|}
		\caption{$T_{i,1}$ for $1 \leq i \leq n+1$.}	
		\label{tab:lmidq-1T_i1}
		\\  \cline{2-6}
		
		&&
		\begin{tabular}{c}
			\textbf{$I_2$ }
		\end{tabular}
		& 
		\begin{tabular}{c}
			$\mathbf{d}(a)$ \, ($a \in \Gamma_{\ell'}$) 
		\end{tabular}
		&
		\begin{tabular}{c}
			$\mathbf{d}'(\xi)$ \, ($ \xi \in \Gamma'_{\ell'}$) 
		\end{tabular}
		&
		\begin{tabular}{c}
			\textbf{$u$}
		\end{tabular}
		\\ \hline \hline

		\multirow{4}{*}[0ex]
		{
			$T_{1,1}$}
		& $1_G + \Xi$
		& $1 + (q+1)\pi_q$
		& $1 + 2\pi_q$
		& $1$
		& $1 + \pi_q$
		\\ \cline{2-6}

		&$\St + \, \Xi$
		& $q + (q+1)\pi_q$
		& $1 + 2\pi_q$
		& $-1$
		& $\pi_q$
		\\ \cline{2-6}
		

		&\begin{tabular}{c}
			$R(\alpha) + \Xi_{\alpha}$ ~ \small{$\left(\alpha \in [T_{\ell'}^\wedge / \equiv], \alpha \neq 1\right)$} 
			
		\end{tabular}
		& $(q+1)(1 + 2\pi_q)$
		& $(\alpha(a) + \alpha(a^{-1}))(1 + 2\pi_q)$
		& $0$
		& $1 + 2\pi_q$
		\\ \cline{2-6}
		

		&\begin{tabular}{c}
			$R'(\theta)$ ~ \small{$\left(\theta \in [T_{\ell'}'^\wedge / \equiv], \theta \neq 1\right)$}
			
		\end{tabular}
		& $q-1$
		& $0$
		& $-\theta(\xi) - \theta(\xi^{-1})$
		& $-1$
		\\ 
		\hline
		\hline

		\multirow{3}{*}[0ex]{
			\begin{tabular}{c}
				$T_{ i,1}$\\
				$(1 \leq i \leq  n)$
			\end{tabular}}
		&	$1_G + \Xi_{i-1}$
		& $1 + (q+1)\pi_{q,i-1}$
		& $1 + 2\pi_{q,i-1}$
		& $1$
		& $1 + \pi_{q,i-1}$
		\\ \cline{2-6}

		&		
		$\St +\,  \Xi_{i-1}$
		& $q + (q+1)\pi_{q,i-1}$
		& $1 + 2\pi_{q,i-1}$
		& $-1$
		& $\pi_{q,i-1}$
		\\ \cline{2-6}


		&\begin{tabular}{c}
			$R(\alpha) + \Xi_{\alpha,i-1}$ ~ \small{$\left(\alpha \in [T_{\ell'}^\wedge / \equiv], \alpha \neq 1\right)$}	

		\end{tabular}
		& $(q+1)(1 + 2\pi_{q,i-1})$
		& $(\alpha(a) + \alpha(a^{-1}))(1 + 2\pi_{q,i-1})$
		& $0$
		& $1 + 2\pi_{q,i-1}$
		\\ 
		\hline  \hline 
		
	\multirow{3}{*}[0ex]{
			$T_{n+1,1}$}
		&$1_G$
		& $1$
		& $1$
		& $1$
		& $1$
		\\ \cline{2-6}

		&$\St$
		& $q$
		& $1$
		& $-1$
		& $0$
		\\ \cline{2-6}


		&\begin{tabular}{c}
			$R(\alpha)$ ~ \small{$\left(\alpha \in [T_{\ell'}^\wedge / \equiv], \alpha \neq 1\right)$}
			
		\end{tabular}
		& $q+1$
		& $\alpha(a) + \alpha(a^{-1})$
		& $0$
		& $1$
		\\ \hline

	\end{longtable}


\renewcommand*{\arraystretch}{1.5}
\begin{longtable}{|c||c|c|c|c|}
	\caption{$T_{i,i}$ for $2 \leq  i \leq n+1$.} 
	\label{tab:lmidq-1T_ii}
	\\  \cline{1-5}
	
	&
	\begin{tabular}{c}
		\textbf{$I_2$ }
	\end{tabular}
	& 
	\begin{tabular}{c}
		$\mathbf{d}(a)$ \,
		($a \in \Gamma_{\ell'}$)
	\end{tabular}
	&
	\begin{tabular}{c}
		\textbf{$\sigma$} 
	\end{tabular}
	\\ \hline \hline 			
	
	$1_G + \Xi_{i-1}$
	& $1 + 2\pi_{q, i-1}$
	& $1 + 2\pi_{q, i-1}$
	& $1$
	\\ \hline

	$\St +  \Xi_{i-1}$
	& $1 + 2\pi_{q, i-1}$
	& $1 + 2\pi_{q, i-1}$
	& $-1$
	\\ \hline

	\begin{tabular}{c}
		$R(\alpha) + \Xi_{\alpha, i-1}$	~ \small{$\left(\alpha \in [T_{\ell'}^\wedge / \equiv], \alpha \neq 1\right)$}	
		
	\end{tabular}
	& $2(1 + 2\pi_{q, i-1})$
	& $(\alpha(a) + \alpha(a^{-1}))(1 + 2\pi_{q, i-1})$
	& $0$
	\\ \hline 
	
	\hline

\end{longtable}

\end{landscape}

\begin{proof}

By Convention~\ref{conv:tsctbl} the labels for the rows of $\Triv_{\ell}(G)$ are the ordinary characters of the trivial source $kG$-modules determined in Lemma~\ref{lem:tsmodslmidq-1}. 
\begin{enumerate}[(a)]
   \item It follows from \cite[Remark~2.5(c)]{BFL22} and Notation~\ref{rem:lmidq-1} that $T_{i,v} = \mathbf{0}$ \smallskip whenever~$v > i$. 
   \item By \cite[Remark~2.5(d)]{BFL22}, the values in  $T_{i,1}$  for $1 \leq i \leq n+1$  (Table~\ref{tab:lmidq-1T_i1}) are calculated by evaluating the character of each trivial source module given in Table \ref{tab:tsmodslmidq-1} at the relevant representatives of the $\ell'$-conjugacy classes of $G$ using the character table of~$G$  \smallskip  (Table \ref{tab:CTBLsl2}). 
\item By Convention~\ref{conv:tsctbl},  the values in $T_{i,i}$ for $2 \leq i \leq n+1$ (Table \ref{tab:lmidq-1T_ii}) are given by the values of the species~$\tau_{Q_{i},s}^{G}$, with $s$ running through  $[\overline{N}_{i}]_{\ell'}$ (identified here with $[N]_{\ell'}$), evaluated at the  trivial source modules $[M]\in\TS(G;Q_{i})$.  By definition of the species and~\cite[Proposition~2.2(d)]{BFL22} these are calculated by evaluating the ordinary character of the $kN$-Green correspondent given in Table~\ref{tab:tsmodsN} of the trivial source $kG$-module labelling the relevant row, at the representatives of the $\ell'$-conjugacy classes of $N$ using the character table of $N$ given  \smallskip  in~Table~\ref{tab:CTBLN}.
    \item    For  each $2 \leq v \leq i \leq n+1$, by Convention~\ref{conv:tsctbl}, the matrix $T_{i,v}$ consists of the values of the species~$\tau_{Q_{v},s}^{G}$, with $s$ running through  $[\overline{N}_{v}]_{\ell'}$  (identified here with $[N]_{\ell'}$), evaluated at the trivial source modules $[M]\in\TS(G;Q_{i})$. However, by definition of the species, $\tau_{Q_{v},s}^{G}([M])=\chi_{\widehat{M[Q_v]}}(s)$ and  \cite[Lemma~2.8]{BFL22} together with Lemma~\ref{lem:omnibus}(a) show that $M[Q_v]$ is the $kN$-Green correspondent of $M$. Hence $T_{i,v} = T_{i,i}$ for all $2 \leq v < i \leq n+1$. 
\end{enumerate}
\end{proof}


\section{Trivial source character table of $G$ when $\ell \mid q+1$}
\label{sec:tschartablelmidq+1}

\begin{nota}\label{rem:lmidq+1}
	We now adopt notation analogous to Notation~\ref{rem:lmidq-1}, in order to describe $\Triv_{\ell}(G)$ according to Convention~\ref{conv:tsctbl}. Here, we fix $Q_{n+1}:=S'_{\ell}\cong C_{\ell^{n}}$. Then, as before, for each $1\leq i\leq n$ we let $Q_{i}$ denote the unique cyclic subgroup of $Q_{n+1}$ of order $\ell^{i-1}$ and  
	\[\{1\} = Q_1 \leq \dots  \leq Q_{n+1} \in \Syl_{\ell}(G)\]
	is our fixed set of representatives for the conjugacy classes of $\ell$-subgroups of $G$.  We keep the same set of representatives for  the $\ell'$-conjugacy classes of $G$:
	$$[G]_{\ell'}:=\{I_{2}\}\cup\{u\} \cup \{\mathbf{d}(a)\mid a\in\Gamma_{\ell'}\}\cup  \{\mathbf{d}'(\xi)\mid \xi\in\Gamma'_{\ell'}\}\,,$$
	where $\Gamma_{\ell'}$ and $\Gamma'_{\ell'}$ are as defined in Notation~\ref{rem:lmidq-1}. Note that in this case $\Gamma_{\ell'} = \Gamma$ as $\ell \nmid q-1$. We fix the following set of representatives for the $\ell'$-conjugacy classes of $N'$: 
	$$[N']_{\ell'}:=\{I_{2}\}\cup \{ \sigma'\} \cup \{\mathbf{d}'(\xi)\mid \xi\in\Gamma'_{\ell'}\}\,.$$
	By the same arguments as in Notation~\ref{rem:lmidq-1}, for any $2\leq v\leq n+1$ and any $1\leq i\leq n+1$ we can label the columns of $T_{i,v}$ by this fixed set of representatives for  the $\ell'$-conjugacy classes of $N'$. Finally, for each $0\leq  i \leq n$ we fix
	\[
	\pi'_{q,i} :=\frac{(q+1)_\ell \cdot(1- \ell^{-i})}{2}\,, \qquad
	\pi^{''}_{q,i} := \frac{(q+1)_{\ell} - 1}{2} - \pi'_{q, i} = \frac{(q+1)_\ell \cdot\ell^{-i} - 1}{2}\,,
	\]
	and let $\pi'_q := \pi'_{q,n}$. 	Again, these numbers arise naturally in the classification of the trivial source modules in blocks with cyclic defect goups in \cite{HL20}.	
\end{nota}

\vspace{2mm}
\subsection{The $\ell$-blocks and trivial source characters of $G$}
\label{subsec:Glmidq+1}

\begin{lem} {\label{lem:blockslmidq+1}}
	When $\ell \mid q+1$ the $\ell$-blocks of $G$, their defect groups and their Brauer trees with type function are as given in Table~\ref{tab:blockslmidq+1}.
\end{lem}

\begin{longtable}{c||c|c|c}
	\caption{The $\ell$-blocks of $\SL_2(q)$ when $\ell \mid q+1$.}	
	\label{tab:blockslmidq+1}
	\\  
	Block $\bB$
	&	\begin{tabular}{c}
		Number of Blocks \\
		(Type)
	\end{tabular}
	& 	\begin{tabular}{c}
		Defect \\
		Groups
	\end{tabular}
	& 
	\begin{tabular}{c}
		Brauer Tree with Type\\
		Function or  $\Irr(\bB)$
	\end{tabular}
	\\ \hline \hline 
	
	$\mathbf{B}_0(G)$ 
	& 	\begin{tabular}{c}
		1 \\
		\small{(Principal)}
	\end{tabular}
	& 	$C_{\ell^n}$ 
	&	\begin{tabular}{c}
		$$ \xymatrix@R=0.0000pt@C=30pt{	
			{_+} & {_-}& {_+}\\
			{\Circle} \ar@{-}[r]  
			& {\Circle} \ar@{-}[r] 
			&{\CIRCLE} \\
			{^{1_G}}&{^{\St}}&{^{\Xi'}}
		}$$ \\
		{\footnotesize $\Xi':= \sum\limits_{\eta \in [S_{\ell}'^\wedge/ \equiv], \eta \neq 1} R'(\eta)$}\\
	\end{tabular}
	\\ \hline 
	
	\begin{tabular}{c}
		$A'_\theta$  \\
		\small{($\theta \in [T_{\ell'}'^\wedge / \equiv], \theta \neq 1$) }
	\end{tabular}
	& \begin{tabular}{c}
		$\frac{(q+1)_{\ell'} - 1}{2}$ \\
		\small{(Nilpotent)}
	\end{tabular}
	& $C_{\ell^n}$ 		
	& 	\begin{tabular}{c}
		$$  \xymatrix@R=0.0000pt@C=50pt{	
			{_-} & {_+}\\
			{\Circle}  \ar@{-}[r]
			& {\CIRCLE}    \\
			{^{R'(\theta)}}&{^{\Xi'_{\theta}}}
		} $$ \\
		{\footnotesize $\Xi'_{\theta}:= \sum\limits_{\eta \in S_{\ell}'^\wedge \setminus \{1\}} R'(\theta \eta)$}\\
	\end{tabular}
	\\ \hline 
	
	\begin{tabular}{c}
		$A_\alpha$ \\
		\small{($\alpha \in [T_{\ell'}^\wedge / \equiv], \alpha \neq 1$)} 
	\end{tabular}
	& 	\begin{tabular}{c}
		$\frac{(q-1)_{\ell'} - 1}{2} = \frac{q-2}{2}$ \\
		\small{(Defect zero)}
	\end{tabular}
	& $\{1\}$ 
	& $\Irr(A_\alpha) = \{R(\alpha)\}$	
	\\ \hline 	
\end{longtable}

\vspace{10mm}


\begin{proof}
	As in Lemma~\ref{lem:blockslmidq-1}, all of the information in the table comes directly from {\cite[Section~II]{Burkhardt}} and the character table of $G$ (Table~\ref{tab:CTBLsl2}), except for the type functions on the Brauer trees, which we compute according to (2) in \S\ref{ssec:cycbl}.  
	 The trivial character is once again positive so the type function for the principal block is immediate, and for each block $A'_{\theta}$, the character $R'(\theta)$ takes a negative value on all non-trivial $\ell$-elements, and is therefore negative. 
\end{proof}

\begin{lem}
	\label{lem:verticeschainandtsmodslmidq+1}
	When $\ell \mid q+1$  the ordinary characters $\chi_{\widehat{M}}$ of the trivial source $kG$-modules~$M$  are as given in Table~\ref{tab:tsmodslmidq+1}, where for each $1 \leq i < n$, 
	\[ \Xi'_{i}   := \sum\limits_{j=1 }^{\pi'_{q,i}} R'(\eta_j)  
	\]
	is a sum of  $\pi'_{q,i}$ pairwise distinct exceptional characters  in $\Irr(\bB_{0}(G))$ and  for any non-trivial $\theta \in \Irr(T'_{\ell'})$,  
	\[\Xi'_{\theta, i}   := \sum\limits_{j=1}^{2\pi'_{q,i} } R'(\theta\eta_j) \]		
	is a sum of $2\pi'_{q,i}$ pairwise distinct exceptional characters  in $\Irr(A'_{\theta})$.
\end{lem}

\newpage

\begin{longtable}{|c|c|c|}
	\caption{Trivial source characters of $\SL_2(q)$ when $\ell \mid q+1$.}	
	\label{tab:tsmodslmidq+1}
	\\  \hline 
	\begin{tabular}{c}
		Vertices of $M$
	\end{tabular}
	& \begin{tabular}{c}
		Character $\chi_{\widehat{M}}$
	\end{tabular}
	& \begin{tabular}{c}
		Block containing $M$
	\end{tabular}
	\\ \hline \hline 
	
	\multirow{3}{*}{$\{1\}$}
	&	$1_G + \St$, $\St + ~ \Xi'$ 
	&
	$\mathbf{B}_0(G)$
	\\
	
	&$R'(\theta) + \Xi'_{\theta}$
	&$A'_{\theta}$ ($\theta \in [T_{\ell'}'^\wedge / \equiv], \theta \neq 1$)
	\\

	&$R(\alpha)$
	&$A_{\alpha}$ ($\alpha \in [T_{\ell'}^\wedge / \equiv], \alpha \neq 1$)
	\\
	
	\hline \hline 
	
	\multirow{2}{*}{
		\begin{tabular}{c}
			$C_{\ell^i}$ \\
			($1 \leq i < n$)
	\end{tabular}}
	&$1_G + \St +  ~ \Xi'_i$, 	
	$\St + ~ \Xi'_i$			
	&$\mathbf{B}_0(G)$
	\\

	&$\Xi'_{\theta, i}$ 
	&$A'_{\theta}$ ($\theta \in [T_{\ell'}'^\wedge / \equiv], \theta \neq 1$) 
	\\
	\hline \hline
	
	\multirow{2}{*}{
		$C_{\ell^n}$}
	&$1_G$, $\Xi'$
	&$\mathbf{B}_0(G)$
	\\
		
	&$\Xi'_{\theta}$
	&$A'_{\theta}$ ($\theta \in [T_{\ell'}'^\wedge / \equiv], \theta \neq 1$)
	\\
	\hline 	
\end{longtable}

\begin{proof}
The arguments are analogous to those given in the proof of Lemma~\ref{lem:tsmodslmidq-1}.
In this case, the parameters (1), (2) and (3) of~\S\ref{ssec:cycbl} necessary to apply  the classification of the trivial source modules given in~\cite[Theorem~5.3]{HL20} are: the Brauer trees with their type function given in~Table~\ref{tab:blockslmidq+1}, and the module $W(\bB)$, which is also always trivial in this case by~\cite[Proposition~6.5(a)]{HL22}. The characters $\chi^{}_{\widehat{M}}$ are then obtained exactly as in the proof of~\cite[Lemma~4.3]{BFL22}. 
\end{proof}

\subsection{The $\ell$-blocks and trivial source characters of $N'$}
\label{subsec:N'}

\begin{lem}
	\label{lem:blocksN'}
	When $\ell \mid q+1$ the blocks of $N'$, their defect groups and their Brauer trees with type function are as given in Table \ref{tab:blocksN'}. 
\end{lem}

\enlargethispage{10mm}

\begin{proof}
	The distribution of the characters of $N'$ into blocks can be determined by examining the values of the central characters of $N'$  modulo $\ell$: 
	\begin{itemize}
		\item the principal block $\bB_0(N')$ contains $1_{N'}$, $\varepsilon'$ and $\chi'_{\eta}$ for each $\eta \in S_{\ell}^{'\wedge} \setminus \{1\}$; and 
		\item for each non-trivial $\theta \in \Irr(T_{\ell'})$, there is a block $\bb'_\theta$ containing $\chi'_{\theta \eta}$ for all $\eta \in \Irr(S_{\ell}^{'})$.		
	\end{itemize} 
The Brauer trees and their type functions are determined using arguments analogous to those in Lemma \ref{lem:blocksN} where in this case we note that $1_{N'}$, $\varepsilon'$ and $\chi'_{\theta}$ ($\theta \in \Irr(T_{\ell'}')$) are $\ell$-rational characters and therefore cannot be exceptional.  	
\end{proof}
\begin{longtable}{c||c|c|c}
	\caption{The $\ell$-blocks of $N'$ when $\ell \mid q+1$ and $q$ is even.}	
	\label{tab:blocksN'}
	\\  
	\begin{tabular}{c} 
		Block \\
		$\bb$
	\end{tabular}
	&	
	\begin{tabular}{c}
		Number of Blocks \\
		(Type)
	\end{tabular}
	& 	\begin{tabular}{c}
		Defect \\
		Groups
	\end{tabular}
	& 
	\begin{tabular}{c}
		Brauer Tree $\sigma(\bb)$ \\
		with Type Function
	\end{tabular}
	\\ \hline \hline 			
	
	$\bB_0(N')$ 
	& \begin{tabular}{c}
		$1$ \\
		(Principal)
	\end{tabular} 
	& 	$C_{\ell^n}$ 
	&	\begin{tabular}{c}
		$$ \xymatrix@R=0.0000pt@C=30pt{
			{_+} & {_-}& {_+}\\
			{\Circle} \ar@{-}[r]
			& {\CIRCLE} \ar@{-}[r] 
			&{\Circle} \\
			{^{1_{N'}}}&{^{\Xi^{N'}}}&{^{\varepsilon'}}
		}$$\\
		{\footnotesize $\Xi^{N'}:= \sum\limits_{\eta \in [S_{\ell}'^\wedge / \equiv], \eta \neq 1} \chi'_{\eta}$}\\
	\end{tabular}
	\\ \hline

	\begin{tabular}{c}
		$\bb'_{\theta}$ \\
		($\theta \in [{T'_{\ell'}}^\wedge / \equiv], \theta \neq 1$)
	\end{tabular}
	&
	\begin{tabular}{c}
	$\frac{(q+1)_{\ell'} -1}{2}$ \\
	(Nilpotent)
	\end{tabular}
	& 	$C_{\ell^n}$ 		
	& 	\begin{tabular}{c}
		$  \xymatrix@R=0.0000pt@C=30pt{	
			{_+} & {_-}\\
			{\Circle} \ar@{-}[r] 
			& {\CIRCLE}    \\
			{^{\chi'_\theta}}&{^{\Xi^{N'}_{\theta}}}
		}$\\
		{\footnotesize $\Xi^{N'}_{\theta}:= \sum\limits_{\eta \in S_{\ell}'^\wedge \setminus \{1\}} \chi'_{\theta \eta}$}\\
	\end{tabular}
	\\ \hline 
	
\end{longtable}

\vspace{2mm}

\enlargethispage{2cm}

\begin{lem}
	\label{lem:tsmodsN'}
	When $\ell \mid q+1$ the ordinary characters $\chi_{\widehat{f(M)}}$ of the $kN'$-Green correspondents $f(M)$ of the trivial source $kG$-modules $M$ with a non-trivial vertex are as given in Table~\ref{tab:tsmodsN'}, where, for each $1 \leq i < n$, 
	\[ \Xi^{N'}_i   := \sum\limits_{j=1}^{\pi^{''}_{q,i}} \chi'_{\eta}
	\]
	is a sum of $\pi^{''}_{q,i}$ pairwise distinct exceptional characters in $\Irr(\bB_0(N'))$, and, for any non-trivial $\theta \in \Irr(T'_{\ell'})$, 
	\[\Xi^{N'}_{\theta, i}   := \sum\limits_{j=1}^{2\pi^{''}_{q,i}} \chi'_{\theta\eta} \]
	is a sum of $2 \pi^{''}_{q,i}$ pairwise distinct exceptional characters of $\Irr(\bb'_\theta)$. 
\end{lem}

\begin{longtable}{|c|l|l|}
	\caption{The trivial source characters of the $kN'$-Green correspondents when $\ell \mid q+1$.}	
	\label{tab:tsmodsN'}
	\\  \hline 
	\begin{tabular}{c}
		Vertices of $M$
	\end{tabular}
	& \begin{tabular}{c}
		Character $\chi_{\widehat{M}}$
	\end{tabular}
	& \begin{tabular}{c}
		Character $\chi_{\widehat{f(M)}}$ of the \\
		Green correspondent 
	\end{tabular}
	\\ \hline \hline

	
	\begin{tabular}{c}
		$C_{\ell^i}$ \\
		($1 \leq i < n$)
	\end{tabular} 
	& 	\begin{tabular}{l}
		$1_G + \St +\, \Xi'_i$ \\ 
		$\St +\, \Xi'_i$ \\ 
		$\Xi'_{\theta, i}$ ($\theta \in [{T'_{\ell'}}^\wedge / \equiv], \theta \neq 1$) 
	\end{tabular}
	& \begin{tabular}{l}
		$1_{N'} + \Xi^{N'}_i$\\ 
		$\varepsilon' + \Xi^{N'}_i$\\ 
		$\chi'_\theta + \Xi^{N'}_{\theta, i}$ 
	\end{tabular}
	
	\\ 	\hline \hline 
	
	$C_{\ell^n}$
	& 
	\begin{tabular}{l}
		$1_G$ \\
		$\Xi'$ \\
		$\Xi'_{\theta}$ ($\theta \in [T_{\ell'}'^\wedge / \equiv], \theta \neq 1$)
	\end{tabular}
	&
	\begin{tabular}{l}
		$1_{N'}$ \\
		$\varepsilon'$ \\
		$\chi'_\theta$
	\end{tabular}
	\\  \hline 	
	
\end{longtable}

\begin{proof}
This proof is analogous to the proof of Lemma~\ref{lem:tsmodsN}. We first determine the Brauer correspondents in $N'$ of the nilpotent blocks of $G$. 
Fix a non-trivial $\theta \in \Irr(T_{\ell'}')$. 
Let $i'_{\theta}$ denote the central primitive idempotent of $\mathcal{O}G$ such that $\mathcal{O}Gi'_\theta$ is the block of $\mathcal{O}G$ corresponding to the $\ell$-block $A_\theta$, and let $i_\theta^{N'}$ denote the central primitive idempotent of $\mathcal{O}N'$ such that $\mathcal{O}N'i^{N'}_\theta$ is the block of $\mathcal{O}N'$ corresponding to the $\ell$-block $\bb'_\theta$. As in  Lemma~\ref{lem:tsmodsN}, we need only compare the coefficients of elements $\mathbf{d}'(\xi^{-1}) \in T'$ in $i'_\theta$ and $i_\theta^{N'}$.
For any $\xi \in \mu_{q+1}$ with non-trivial $\ell$-part, the coefficient of $\mathbf{d}'(\xi^{-1}) \in T'$ is $0$ in both $i'_\theta$ and $i_\theta^{N'}$. If $\xi$ is an $\ell'$-element, then the coefficient of  $\mathbf{d}'(\xi^{-1}) $ in $i'_\theta$ is
\[ 
\left\{ 
\begin{array}{ll}
 \frac{(q-1)^2|S_{\ell}'|}{|G|}
 = \frac{q-1}{q|T'_{\ell'}|}
 & \mbox{for } \xi = 1,\\
 \frac{(q-1)|S_{\ell}'|}{|G|}\left(-\theta(\xi) - \theta (\xi^{-1}) \right)  
 =  \frac{1}{q|T'_{\ell'}|}\left(-\theta(\xi) - \theta (\xi^{-1}) \right)  
 & \mbox{for } \xi \neq 1, \xi \in (\mu_{q+1})_{\ell'},
 \end{array}
 \right.
 \]
and the coefficient in $i_\theta^{N'}$ is
\[ 
\left\{ 
\begin{array}{ll}
 \frac{4|S_{\ell}'|}{|N'|}
 = \frac{2}{|T'_{\ell'}|}
 & \mbox{for } \xi = 1,\\
 \frac{2|S_{\ell}'|}{|N'|} \left(\theta(\xi) + \theta (\xi^{-1}) \right)  
 =  \frac{1}{|T'_{\ell'}|}\left(\theta(\xi) + \theta (\xi^{-1}) \right)  
 & \mbox{for } \xi \neq 1, \xi \in (\mu_{q+1})_{\ell'}.
 \end{array}
 \right.
 \]
Since $\ell \mid q+1$ we have $\frac{q-1}{q} \equiv 2 \mod \mathfrak{p}$ and $\frac{1}{q} \equiv -1 \mod \mathfrak{p}$, and therefore $\Br_{S_{\ell}'}(\overline{i'_\theta}) \equiv \overline i_\theta^{N'} \mod \mathfrak{p}$ so $\bb'_\theta$ is the Brauer correspondent of $A'_\theta$  in $N'$. 

As mentioned in Lemma~\ref{lem:tsmodsN}, the Green correspondent of a trivial source $kG$-module in a block~$\bB$ of $G$ lies in the Brauer correspondent block $\bb$ of $N'$, and since $W(\bB)=W(\bb)$ is trivial in all cases, the characters of the trivial source $\bb$-modules can be determined using \cite[Theorem~5.3]{HL20}. Moreover, having determined the Brauer correspondent blocks, in each case there is only one possible choice for the Green correspondent of a trivial source module of $kG$ with a fixed vertex. The characters of these Green correspondent modules are as in Table\ ~\ref{tab:tsmodsN'}. 	
\end{proof}

\subsection{The trivial source character table of $G$}

\vspace{2mm}

\begin{thm}\label{thm:l|q+1}
Let $G= \SL_2(q)$ with $q = 2^f$ for some integer $f \geq 2$, and suppose that ${\ell \mid q+1}$. 
Then,  with notation as in Notation~\ref{rem:lmidq+1}, the trivial source character table $\Triv_{\ell}(G)= [T_{i,v}]_{1\leq i,v\leq n+1}$ is given as follows:
       \begin{enumerate}[{\,\,\rm(a)}] \setlength{\itemsep}{2pt}
		\item $T_{i,v} = \mathbf{0}$ if $v > i$;
		\item the matrices $T_{i,1}$ are as given in Table \ref{tab:lmidq+1T_i1} for each $1 \leq i \leq n+1$; 
		\item the matrices $T_{i,i}$ are as given in Table \ref{tab:lmidq+1T_ii} for each $2 \leq i \leq n$;  
		\item the matrix $T_{n+1,n+1}$ is as given in Table \ref{tab:lmidq+1T_n+1n+1}; and 
		\item $T_{i,v} = T_{i,i}$ for all $2 \leq v \leq i \leq n+1$.
	\end{enumerate} 	
\end{thm}

\begin{proof}
	The calculations are completely analogous to those in the proof of Theorem~\ref{thm:l|q-1} except that we take the trivial source $kG$-modules from Table~\ref{tab:tsmodslmidq+1}, their  Green correspondents from Table~\ref{tab:tsmodsN'} and we use the character table of $N'$ from~Table~\ref{tab:CTBLN'}.	
\end{proof}

\begin{landscape}
\thispagestyle{plain}
	
	\renewcommand*{\arraystretch}{1.5}
	\begin{longtable}{c|c||c|c|c|c|}
		\caption{$T_{i,1}$ for $1 \leq i \leq n+1$.}	
		\label{tab:lmidq+1T_i1}
		\\  \cline{2-6}
		
		&&
\begin{tabular}{c}
	$I_2$
\end{tabular}
& 
\begin{tabular}{c}
	$\mathbf{d}(a)$ \vspace{-.8ex}  \\
	({\footnotesize $ a \in \Gamma_{\ell'}$})
\end{tabular}
&
\begin{tabular}{c}
	$\mathbf{d}'(\xi)$ \vspace{-.8ex}  \\
	({\footnotesize$\xi \in \Gamma'_{\ell'}$})
\end{tabular}
&
\begin{tabular}{c}
	$u$
\end{tabular}
\\ \hline \hline

		\multirow{4}{*}[0ex]{
			$T_{1,1} $}
		& $1_G + \St$ 
		& $1 + q$ 
		& $2$ 
		& $0$ 
		& $1$ 
		\\ \cline{2-6}

		& $\St + \, \Xi'$ 
		& $ q + (q-1)\pi'_q$ 
		& $1$ 
		& $-1 - 2\pi'_q$ 
		& $-\pi'_q $
		\\ \cline{2-6}

		&\begin{tabular}{c}
			$R'(\theta) + \Xi'_{\theta}$ \, \small{$\left(\theta \in [T_{\ell'}'^\wedge / \equiv], \theta^2 \neq 1\right)$}
			
		\end{tabular}
		& $(q-1)(1 + 2\pi'_q)$
		& $0$
		& $-(\theta(\xi) + \theta(\xi^{-1}))(1 + 2\pi'_q)$
		& $-(1+ 2\pi'_q)$
		\\ \cline{2-6}

		&\begin{tabular}{c}
			$R(\alpha)$ \, \small{$\left(\alpha \in [T_{\ell'}^\wedge / \equiv], \alpha^2 \neq 1\right)$} 
			
		\end{tabular}
		& $q+1$
		& $\alpha(a) + \alpha(a^{-1})$
		& $0$
		& $1$
		\\ 
		\hline
		\hline

		\multirow{3}{*}[0ex]{
			\begin{tabular}{c}
			$T_{ i,1}$\\
			$(2 \leq i \leq n)$
		\end{tabular}}
		& $1_G + \St + \, \Xi'_{i-1}$ 
		& $1 + q + (q-1)\pi'_{q,i-1}$
		& $2$ 
		& $-2\pi'_{q,i-1}$ 
		& $1 - \pi'_{q,i-1}$
		\\ \cline{2-6}

		&		
		$\St +\, \Xi'_{i-1}$ 
		& $q + (q-1)\pi'_{q,i-1}$
		& $1$
		& $-1-2\pi'_{q,i-1}$
		& $-\pi'_{q,i-1}$
		\\ \cline{2-6}

		&\begin{tabular}{c}
			$\Xi'_{\theta, i-1}$\,	\small{$\left(\theta \in [T_{\ell'}'^\wedge / \equiv], \theta \neq 1\right)$} 
			
		\end{tabular}
		& $2(q-1)\pi'_{q,i-1}$
		& $0$
		& $-2\left(\theta(\xi) + \theta(\xi^{-1})\right)\pi'_{q,i-1}$ 
		& $-2\pi'_{q,i-1}$
		\\ 
		\hline 		
		\hline 
		
		\multirow{3}{*}[0ex]{
			$T_{n+1,1}$}
		& $1_G$ 
		& $1$
		& $1$
		& $1$
		& $1$
		\\ \cline{2-6}

		& $\Xi'$ 
		& $(q-1)\pi'_q$
		& $0$ 
		& $-2\pi'_q$
		& $-\pi'_q$
		\\ \cline{2-6}

		&\begin{tabular}{c}
			$\Xi'_{\theta}$ \, \small{$\left(\theta \in [T_{\ell'}'^\wedge / \equiv], \theta \neq 1\right)$}
			
		\end{tabular}
		& $2(q-1)\pi'_q$
		& $0$
		& $-2\pi'_q(\theta(\xi) + \theta(\xi^{-1}))$
		& $-2\pi'_q$
		\\ \hline
		
	\end{longtable}

\begin{center}
\begin{minipage}{.4\textwidth} 
\renewcommand*{\arraystretch}{1.5}
\begin{longtable}{|c||c|c|c|c|}
	\caption{$T_{i,i}$ for $2 \leq  i \leq n$.} 
	\label{tab:lmidq+1T_ii}
	\\  \cline{1-5}
	
	&
	\begin{tabular}{c}
		\textbf{$I_2$ }
	\end{tabular}
	& 
	\begin{tabular}{c}
		$\mathbf{d}'(\xi)$ \, ($ \xi \in \Gamma'_{\ell'}$)
	\end{tabular}
	&
	\begin{tabular}{c}
		\textbf{$\sigma'$} 
	\end{tabular}
	\\ \hline \hline 			
	
	$1_G + \St +\, \Xi'_{i-1}$ 
	& $1 + 2\pi''_{q,i-1}$ 
	& $1 + 2\pi''_{q,i-1}$
	& $1$ 
	\\ \cline{1-5}

	$\St +\, \Xi'_{i-1}$ 
	& $1 + 2 \pi''_{q,i-1}$
	& $1 + 2\pi''_{q,i-1}$
	& $-1$ 
	\\ \cline{1-5}

	\begin{tabular}{c}
		$\Xi'_{\theta, i-1}$ \vspace{-1ex}\\ 
		\tiny{$\left(\theta \in [T_{\ell'}'^\wedge / \equiv], \theta\neq 1\right)$}
	\end{tabular}
	& $2(1 + 2\pi''_{q,i-1})$
	& $\left(\theta(\xi) + \theta(\xi^{-1})\right)(1 + 2\pi''_{q,i-1})$
	& $0$
	\\ \hline 
	
\end{longtable}
\end{minipage}\hfil
\begin{minipage}[t]{.4\textwidth}
~
\end{minipage}\hfil
\begin{minipage}{.55\textwidth}
\renewcommand*{\arraystretch}{1.5}
\begin{longtable}{|c||c|c|c|c|}
	\caption{$T_{n+1,n+1}$.} 
	\label{tab:lmidq+1T_n+1n+1}
	\\  \cline{1-5}
	
	&
	\begin{tabular}{c}
		\textbf{$I_2$ }
	\end{tabular}
	& 
	\begin{tabular}{c}
		$\mathbf{d}'(\xi)$ \, ($\xi \in \Gamma'$)
	\end{tabular}
	&
	\begin{tabular}{c}
		\textbf{$\sigma'$}
	\end{tabular}
	\\ \hline \hline 			
	
	$1_G$ 
	& $1$ 
	& $1$
	& $1$ 
	\\ \cline{1-5}

	$\Xi'$ 
	& $1$
	& $1$
	& $-1$ 
	\\ \cline{1-5}

	\begin{tabular}{c}
		$\Xi'_{\theta}$ \vspace{-1ex}\\ 
		\tiny{$\left(\theta \in [T_{\ell'}'^\wedge / \equiv], \theta \neq 1\right)$}
	\end{tabular}
	& $2$
	& $\theta(\xi) + \theta(\xi^{-1})$
	& $0$
	\\ \hline 
	
\end{longtable}
\end{minipage}
\end{center}
\end{landscape}

\vspace{4mm}

\textbf{Acknowledgments.}
The authors thank Olivier Dudas for useful discussions and Gunter Malle for feedback on an earlier version of this manuscript. The first author acknowledges the hospitality of the \emph{Lehrstuhl f\"ur Algebra und Zahlentheorie} of the \emph{RWTH Aachen University} during a visit in November 2021 to work on this article. The second author gratefully acknowledges financial support by SFB TRR 195, and thanks the \emph{Institut f\"ur Algebra, Zahlentheorie und Diskrete Mathematik} of the \emph{Leibniz Universit\"at Hannover} for their hospitality during the writing period of this article.


	\nocite{}
	\bibliographystyle{aomalpha}
	\bibliography{biblio.bib}
	\bigskip
	


\end{document}